\documentclass[11 pt]{amsart}

\usepackage{graphicx}
\usepackage{color}
\usepackage{amssymb, amsmath, amsthm}
\usepackage{mathptmx}
\usepackage{comment}

\usepackage{epsfig}
\usepackage{epstopdf}
\usepackage{graphicx}
\usepackage{pinlabel}
\usepackage{enumerate}

\theoremstyle{plain}

\newtheorem{theorem}{Theorem}
\newtheorem{lemma}{Lemma}
\newtheorem{corollary}{Corollary}
\newtheorem{claim}{Claim}

\newtheorem{remark}{Remark}

\newcommand{\bi}{\begin{itemize}}
\newcommand{\ei}{\end{itemize}}
\newcommand{\be}{\begin{enumerate}}
\newcommand{\ee}{\end{enumerate}}

\newcommand{\ob}[1]{\overline{#1}}

\numberwithin{definition}{section}
\numberwithin{example}{section}
\numberwithin{lemma}{section}
\numberwithin{theorem}{section}
\numberwithin{corollary}{section}

\begin{document}
\title{Twist Number and the Alternating Volume of Knots}
\author[H. Allen]{Heidi Allen}
\author[R. Blair]{Ryan Blair}
\author[L. Rodriguez]{Leslie Rodriguez}

\maketitle

\begin{abstract}
It was previously shown by the second author that every knot in $S^3$ is ambient isotopic to one component of a two-component, alternating, hyperbolic link. In this paper, we define the alternating volume of a knot $K$ to be the minimum volume of any link $L$ in a natural class of alternating, hyperbolic links such that $K$ is ambient isotopic to a component of $L$. Our main result shows that the alternating volume of a knot is coarsely equivalent to the twist number of a knot.
\end{abstract}

\section{Introduction}

A core idea in the study of low-dimensional topology and geometry is to develop a dictionary of translation between combinatorial information and geometric information. This approach has been particularly successful in the study of hyperbolic links in $S^3$. In particular, significant progress has been made in relating the hyperbolic volume of links to the combinatorics of link diagrams.

The combinatorial invariant of twist number has been shown to have deep connections to the hyperbolic volume of an alternating, hyperbolic link. In the sphere of projection for a link diagram $D$, a \emph{twist region} is a maximal collection of bigons in the link diagram stacked end to end or a neighborhood of a crossing not incident to any bigon. The integer $t(D)$ denotes the number of twist regions of $D$. Lackenby showed that if a hyperbolic link has a prime alternating diagram $D$, then the hyperbolic volume of that link is coarsely equivalent to $t(D)$ (i.e. the hyperbolic volume is bounded both above and below by a linear function of $t(D)$) \cite{Lackenby}. Hence, for such links, hyperbolic volume is roughly equated to $t(D)$. The twist number of a link $L$ is denoted $t(L)$ and is the minimal value of $t(D)$ over all diagrams $D$ of $L$.

In \cite{Blair1}, the second author produced an algorithm that can be applied to any diagram of any knot $K$ to produce a diagram of an alternating link $L$ such that the projection of $K$ is contained in the projection of $L$. This algorithm together with results of Menasco \cite{Menasco} were combined to show that given any knot $K$ in $S^3$, there exists an unknot in the complement of $K$, denoted $U$, such that $K\cup U$ is an alternating, hyperbolic link. Roughly speaking, we define the \emph{alternating volume} of a knot $K$, denoted $AltVol(K)$, to be the minimum volume of any such hyperbolic alternating link $K\cup U$. The precise definition of alternating volume will be given in the next section. Our main result demonstrates that the alternating volume of a knot is coarsely equivalent to the twist number of a knot. In particular, we show the following:

\begin{theorem}[Main]
Let $V_3$ be the volume of a regular ideal hyperbolic tetrahedron. Given any non-alternating knot $K$,

$$V_3(t(K)-2)\leq AltVol(K)\leq 10V_3(5t(K)-1).$$
\end{theorem}

The above theorem demonstrates that alternating volume is a topologically meaningful method of assigning a hyperbolic structure to a non-hyperbolic knot. Since alternating knot complements have played a special role in the study of hyperbolic 3-manifolds, we hope that further study of the algorithm in \cite{Blair1} and alternating volume will result in additional interesting connections between hyperbolic structures and non-hyperbolic knots.

Our main result is inspired by recent work of Rieck and Yamashita \cite{RY} in which they define the link volume of any closed orientable 3-manifold. However, our techniques are similar to those used in \cite{Blair1} and differ significantly from those used by Rieck and Yamashita. The link volume is a weighted hyperbolic volume assigned to a closed orientable 3-manifold $M$ by viewing $M$ as a cover of $S^3$ branched over a hyperbolic link. Hence, Rieck and Yamashita are able to assign weighted hyperbolic volumes to 3-manifolds in a topologically meaningful way by relating these 3-manifolds to hyperbolic links.

\section{Definitions and Preliminaries}

In this paper, we will use the term \emph{link} to mean a smooth embedding of a disjoint collection of circles into $S^3$. A \emph{link projection} $D$ is the image of the link under a regular projection into a 2-sphere  $S\subset S^3$. We refer to $S$ as the \emph{sphere of projection}. The link projection $D$ is a finite 4-valent graph in $S$. A \emph{link diagram} is a link projection together with two edge labels assigned to every edge which encode crossing information. See figure \ref{fig:4-valentgraph.eps}. Although this is the convention we use in our proofs, when illustrating link projections we use the standard convention in an effort to improved readability. Notice that every edge of $D$ receives two labels. An \emph{alternating edge} is labeled with both a plus and a minus and a \emph{non-alternating edge} is labeled with two plus signs or two minus signs. A diagram $D$ is \emph{alternating} if all edges of $D$ are alternating. Moreover, we say a diagram is non-alternating if it has at least one non-alternating edge. Note that this convention implies that the standard diagram of the unknot is \emph{not} non-alternating.

\begin{figure}[h]
\centering \scalebox{.3}{\includegraphics{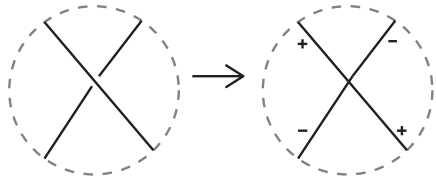}}
\caption{Our labeling convention for link diagrams.}\label{fig:4-valentgraph.eps}
\end{figure}

%Define $P$ to be the standard projection map then $P:(S^3\setminus \{x_1,x_2\})\rightarrow S^2$. Given a knot $K$ in $S^3$, the \emph{projection} of $K$ is the regular image of the restriction of the map $P$ to $K$. A \emph{knot diagram} is a knot projection $P$ together with information at each crossing that indicates which strand goes over and which one goes under. A knot diagram is \emph{alternating} if as we run along the knot diagram and record at each crossing if we went over $(o)$ or under $(u)$ another strand, then the collection of labels alternates between $(o)$ and $(u)$. A knot is \emph{alternating} if it has an alternating diagram.

Recall that a link is \emph{prime} if it cannot be decomposed into a connected sum of non-trivial links. A link is \emph{split} if there exists a 2-sphere embedded in the exterior of the link which separates components of the link. A link that is not split is called \emph{non-split}. A link diagram is \emph{connected} if it is a connected subset of the sphere of projection. We say a link diagram is \emph{prime} if, for every simple closed curve $C$ in the sphere of projection $S$ which intersects $D$ transversely in exactly two points neither of which is a vertex of $D$, $C$ bounds a disk $E$ in $S$ such that $E \cap D$ contains no vertices of $D$. If a link diagram $D$ has no cut vertex, we say $D$ is \emph{reduced}. A link diagram is \emph{$R2$-reduced} if there is no configuration of the form depicted in Figure \ref{fig:R2Move.eps}. Note that every link diagram can be converted to a reduced diagram via a finite sequence of flypes and every link diagram can be converted to a $R2$-reduced diagram by a finite sequence of type II Reidermeister moves.

 \begin{figure}[h]
\centering \scalebox{.2}{\includegraphics{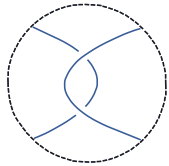}}
\caption{An $R2$-reduced diagram does not contain this configuration.}\label{fig:R2Move.eps}
\end{figure}

Given a link diagram $D$ contained in the sphere of projection $S$, a \emph{region} of $D$ is the closure of a component of $S\setminus D$ in $S$. A \emph{bigon region} is a disk region such that this disk meets $D$ in exactly two edges. Given a link diagram $D$, let $\tau$ be an open regular neighborhood of all bigon regions of $D$ in the sphere of projection. Each connected component of $\tau$ is called a \emph{twist region}. Additionally, neighborhoods of single crossings of $D$ which are not incident to any bigon are also considered twist regions. Hence, every crossing of $D$ is contained in some twist region of $D$. The total number of twist regions in $D$ is the \emph{twist number of $D$} and is denoted \emph{$t(D)$}. The \emph{twist number of a link $L$} is the minimum value of $t(D)$ over all diagrams of $L$ and is denoted \emph{$t(L)$}.

Define a \emph{sub twist region} of a diagram $D$ to be either an open regular neighborhood of exactly one crossing of $D$ or a connected open regular neighborhood of some collection of bigons for $D$. See Figure \ref{fig:SubTwist.eps}.

\begin{figure}[h]
\centering \scalebox{.15}{\includegraphics{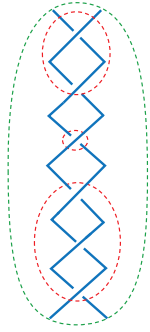}}
\caption{An example of three disjoint sub twist regions contained in a common twist region.}\label{fig:SubTwist.eps}
\end{figure}

A link $K$ in $S^3$ is \emph{hyperbolic} if $S^3-K$ has a complete Riemannian metric of constant sectional curvature equal to $-1$. Due to Thurston's foundational work on hyperbolic $3$-manifolds, it is known that every non-split link in $S^3$ fits into exactly one of three mutually disjoint categories:

\begin{enumerate}
\item hyperbolic links

\item satellite links

\item torus links

\end{enumerate}

%This classification can be taken to be a kind of definition of hyperbolic links. Namely, a hyperbolic link is any link that is non-split, not a satellite link, not a torus link and not the unknot. By Mostow rigidity, the exterior of every hyperbolic link $K$ has a well-defined finite \emph{hyperbolic volume}, denoted $vol(K)$.

%\begin{figure}[h]
%\centering \scalebox{.2}{\includegraphics{Figure8.eps}}
%\caption{An example of a diagram of an alternating, hyperbolic knot.}\label{fig:Figure8.eps}
%\end{figure}

The complement of every hyperbolic link $K$ in $S^3$ has a well-defined hyperbolic volume, denoted $vol(S^3\setminus K)$. Significant effort has been made to develop combinatorial criteria for a link to be hyperbolic. The following result of Menasco shows that most alternating links are hyperbolic.

\begin{theorem}\label{Menasco}\cite{Menasco}
If $L$ is a non-split, prime, alternating link which is not a $(2,q)$ torus
link, then $L$ is hyperbolic.
\end{theorem}

The following theorem shows that every knot can be converted to an alternating link by adding an unknotted component.

\begin{theorem}\label{BlairMain}\cite{Blair1}
Given any connected diagram of a non-alternating link, we can augment
the diagram by adding a single unknotted component so that the resulting link
diagram is alternating.
\end{theorem}

In \cite{Blair1}, Theorem \ref{Menasco} and Theorem \ref{BlairMain} were combined to show the following corollary.

\begin{corollary}\label{LastCorBlair} \cite{Blair1} Every link complement $S^3\setminus K$ contains an unknot $U$ such that $K\cup U$ is a hyperbolic link.

\end{corollary}

Moreover, the proof of Corollary \ref{LastCorBlair} demonstrates the slightly stronger result that for every link $K$ in $S^3$ there is a diagram $D$ of $K$ and an unknot $U$ in $S^3$ such that $U$ projects to a simple closed curve in the sphere of projection for $D$, $K\cup U$ is hyperbolic and $K\cup U$ has an alternating diagram in the sphere of projection for $D$. The link $K\cup U$ is called an \emph{alternating augmentation} of $D$. The unknot $U$ is called the \emph{augmenting component} of $K\cup U$. The diagram of an alternating augmentation will be denoted $G$ and will always denote the diagram of $K\cup U$ achieved by projecting $K\cup U$ onto the sphere of projection for $D$. Hence, $D$ is a subset of $G$ and the closure of $G\setminus D$ is a simple closed curve $A$ in the sphere of projection. Note that $A$ is the projection of $U$. See Figure \ref{fig:nonaltturnedalt.eps}.

\begin{figure}[ht]
\labellist \small\hair 2pt
\pinlabel {$D$} [b] at 140 -20
\pinlabel {$G$} [b] at 450 -20
\pinlabel {$A$} [b] at 390 250
%\pinlabel {$B^*$} [b] at 30 15
%\pinlabel {$\Theta$} [b] at 340 205
\endlabellist
\includegraphics[scale=.5]{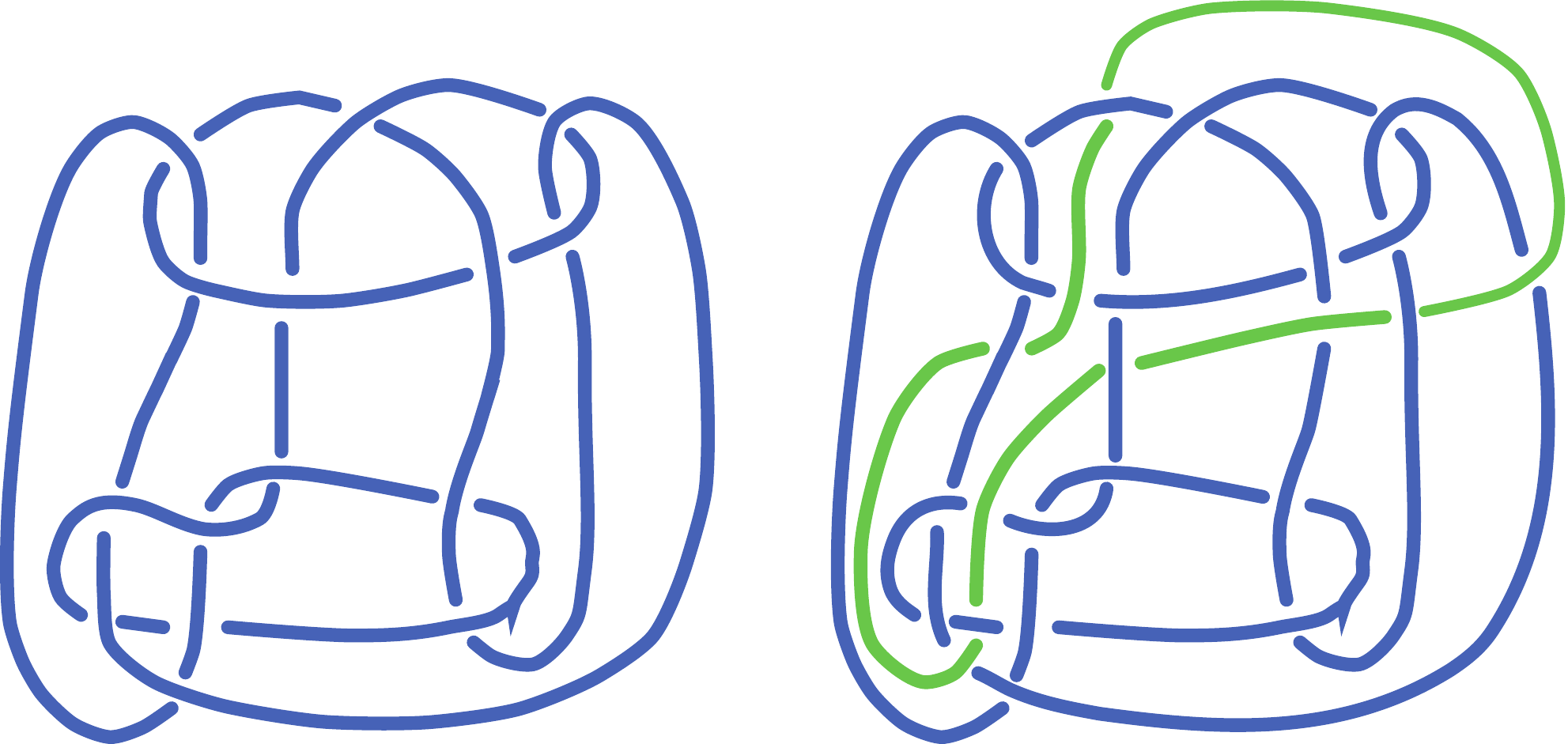}
\caption{Making a non-alternating knot into an alternating link}\label{fig:nonaltturnedalt.eps}
\end{figure}

%\begin{figure}[h]
%\centering \scalebox{.5}{\includegraphics{nonaltturnedaltr.pdf}}
%\caption{Making a non-alternating knot into an alternating link}\label{fig:nonaltturnedalt.eps}
%\end{figure}

We are now able to define the \emph{alternating volume} of a knot $K$, denoted $AltVol(K)$. Given a knot $K$, the alternating volume of $K$ is the infimum of the hyperbolic volumes of all hyperbolic alternating augmentations of $K$. By J\o{}rgensen and Thurston, we know that every non-empty set of hyperbolic volumes attains its infimum. Hence, the infimum in the definition of alternating volume can be replaced with a minimum. Thus,

$$AltVol(K)=\mathop{min}_{U\subset S^3}\{vol(S^3\setminus(K\cup U))\}.$$ where the above minimum is taken over all augmenting components $U$ such that the alternating augmentation $K\cup U$ is hyperbolic.

Given a hyperbolic knot $K$ and any hyperbolic alternating augmentation of $K$, $K\cup U$, then Thurston's hyperbolic Dehn surgery theorem implies that $vol(S^3\setminus(K\cup U))>vol(S^3\setminus K)$. Hence, if $K$ is a hyperbolic knot, it is always true that $AltVol(K)>vol(S^3\setminus K)$.

Agol showed that the Whitehead link complement and the $(-2,3,8)$ pretzel link complement are the minimal volume orientable hyperbolic $3$-manifolds with two cusps \cite{Agol}. Each of these manifolds has hyperbolic volume $4G$ where $G$ is Catalan's constant. Moreover, the standard diagram of the Whitehead link is an alternating augmentation of an unknot diagram. Similarly, the standard diagram of the $(-2,3,8)$ pretzel link is an alternating augmentation of a trefoil diagram. Thus, the alternating volume of both the unknot the trefoil is $4G$. Computing the alternating volume for other knot types will likely be a challenging problem.

The main goal of this article is to relate the alternating volume of a knot to the twist number of a knot in a manner similar to the following theorem due to  Lackenby\cite{Lackenby} with improvements on the theorem due to Agol and D. Thurston.

\begin{theorem}\label{Lackenby}\cite{Lackenby}
Let $V_3$ be the volume of a regular ideal hyperbolic tetrahedron. Let $D$ be a
prime, alternating diagram for a hyperbolic link $L$. Then,

$$V_3(t(D)-2)\leq vol(S^3\setminus L)\leq 10V_3(t(D)-1)$$
\end{theorem}

A portion of the proof of the main theorem will be devoted to proving generalizations of the results in \cite{Blair1} by carefully controlling how the diagram of an augmenting component intersects the diagram of the original knot. We use the remainder of this section to record results from that paper that we will use to control the projection of the augmenting component. Although we make an effort to keep this paper self-contained, there are two instances where we use results from \cite{Blair1} which follow from the proofs but not the statement of the theorems presented there. These instances include the definition of alternating augmentation and Remark 1. In both these cases, we provide references so that the interested reader can verify our claims.

In \cite{Blair1} the second author defines \emph{Type I} and \emph{Type II} moves which are local moves preformed on link diagrams. These moves have the potential to change the link type, but will always preserve the fact that a link diagram is alternating. The Type I move is depicted in Figure \ref{fig:Type1Move.eps} and the Type II move is depicted in Figure \ref{fig:R1Move.eps}.

\begin{lemma}\label{TypeI}\cite{Blair1}
Given an alternating connected diagram of a link, a Type I move
results in an alternating link diagram.
\end{lemma}

\begin{figure}[h]
\centering \scalebox{.3}{\includegraphics{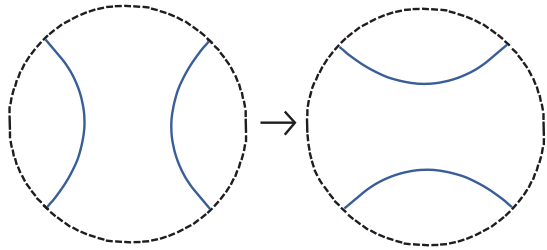}}
\caption{A Type I move.}\label{fig:Type1Move.eps}
\end{figure}

\begin{lemma}\label{TypeII}\cite{Blair1}
Given an alternating connected diagram of a link, a Type II move
(after choosing the signs of the new crossings) results in an alternating link
diagram.
\end{lemma}

\begin{figure}[h]
\centering \scalebox{.3}{\includegraphics{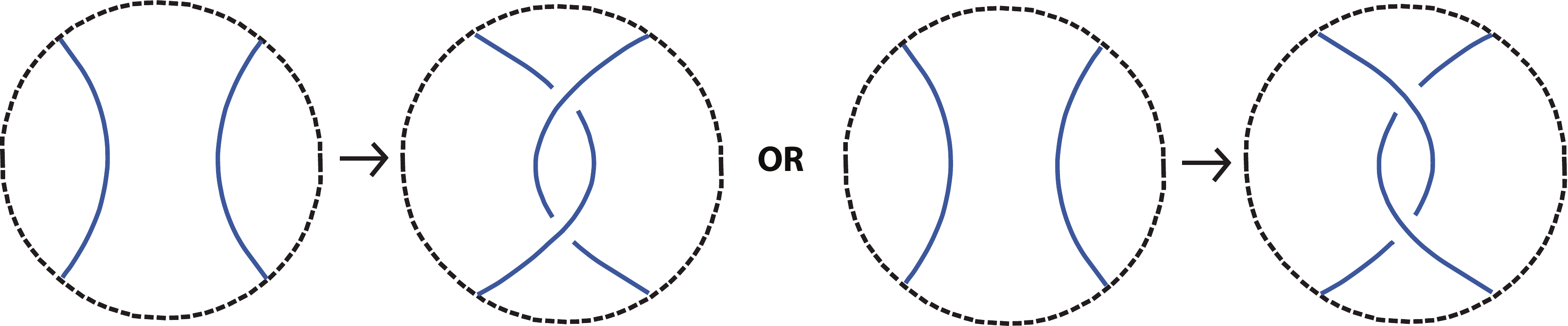}}
\caption{A Type II move.}\label{fig:R1Move.eps}
\end{figure}

\begin{lemma}\label{BlairUnlink}\cite{Blair1}
Any non-alternating connected link diagram can be augmented via an unlink such that it becomes alternating.
\end{lemma}

\begin{remark}\label{UnlinkRemark}
The proof of Lemma \ref{BlairUnlink} appears on page 68 of \cite{Blair1} and some of the details of that proof will become relevant later in this paper. In particular, the proof of Lemma \ref{BlairUnlink} implies that given any connected, non-alternating link diagram $D$ for a link $K$, it is possible to find an unlink $U$ in $S^3$ such that $U$ projects to a collection of disjoint simple closed curves in the sphere of projection for $D$ and the projection of $U$ intersects every non-alternating edge of $D$ exactly once, but is otherwise disjoint from $D$. Moreover, the resulting diagram of $K\cup U$ is alternating.
\end{remark}

%Let $D$ be a connected, non-alternating link projection in $S^2$. Let $A$ be a simple closed curve in the plane of projection which is disjoint from the vertices of $D$ and meets the edges of $D$ transversely. Augment $D$ with $A$ in such a way that the resulting link diagram, $G$, is alternating and the resulting link is hyperbolic. This is possible by corollary 5 of "Alternating Augmentations of Links". Let $K$ be a knot in $S^3$ with projection $D$. Let $K_A$ be an unknot in $S^3$ that augments $K$. Let $L$ be a link in $S^3$ defined to be $K\cup K_A$.

\section{Initial Inequalities}

\begin{lemma}\label{lowerbound}

Let $D$ be any diagram of a knot and let $G$ be the diagram of an alternating augmentation of $D$, then $t(D) \leq t(G)$.

\end{lemma}

\begin{proof}

%Let $B_2$ be the union of all bigons of $G$. Let $B$ be an open regular neighborhood of $B_2$ in $S^2$, the sphere of projection for $D$. Let $C_I$ be an open regular neighborhood of the crossings of $G$ that are not contained in $B$. A twist region of $G$ is a connected component of $C_I \cup B$. Remove $A$ from $G$ and consider $C_I \cup B$ as a collection of disjoint planar surfaces in $S^2$.

%Note that every twist region of $G$ is contained in exactly one of $C_I$ or $B$.

%Suppose that $T$ is a twist region of $G$ in $B_D \cap C_I$. Then $T$ contains exactly one crossing of $D$ and $T$ is a sub twist region of $D$.

Since $D$ is a subset of $G$, every crossing of $D$ is a crossing of $G$. Let $\tau_D$ be the set of twist regions of $G$ that contain at least one crossing of $D$. Suppose that $T\in \tau_D$. The portion of $G$ in $T$ is the projection of two sub arcs of $K\cup A$. Moreover, both of these arcs must be contained in $K$ since $T$ contains a crossing of $D$. Hence, $T$ is a sub twist region of $D$. Thus, all elements of $\tau_D$ are sub twist regions of $D$.

%Since $A$ projects to a simple closed curve in $S^2$, no twist region of $G$ is disjoint from $D$. Now we examine the twist regions of $G$.

%Let $t_D=|B_D|$ be the number of twist regions of $G$ that meet crossings of $D$ and let $t_A=|C_I\cup B|-|B_D|$ be the number of twist regions of $G$ that do not meet crossings of $D$.  As previously demonstrated, each element of $B_D$ is a sub twist region of $D$.

Let $P$ be the partition of the set of crossings of $D$ corresponding to the twist regions of $D$. Let $P'$ be the partition of the set of crossings of $D$ corresponding to $\tau_D$. To show $P'$ is a refinement of $P$, we will show that each of the following must hold:

\begin{enumerate}

\item Distinct elements of $P'$ intersect trivially.

\item For all $x \in P'$ there exists $y \in P$ such that $x \subset y$.

\item  $\bigcup_{x \in P'}x$ is the set of all crossings of $D$.

\end{enumerate}

By definition, all of the twist regions of $G$ are pairwise disjoint. Thus, distinct elements of $P'$ intersect trivially.

As argued above, all twist regions in $\tau_D$ are sub twist regions of $D$. Thus, for all $x \in P'$ there exists a $y \in P$ such that $x \subset y$.

Since every crossing of $D$ is contained in some twist region in $\tau_D$, then $\bigcup_{x\in P'}x$ is the set of all crossings of $D$.

Therefore, $P'$ is a refinement of $P$.

Since $P'$ is a refinement of $P$, then $|P|\leq|P'|$. Since $t(D)=|P|$ and $|\tau_D|=|P'|$, then

$$t(D)\leq |\tau_D| \leq t(G).$$

\end{proof}

%\begin{lemma}\label{disjfromregion}

%M doesn't pass through twist region of D (awaiting approval)

%\end{lemma}

%

%\begin{proof}

%\end{proof}

%\begin{lemma}

%The induction lemma. (write this up!)

%\end{lemma}

%\begin{proof}

%\end{proof}

The following results will be need in the proof of Lemma \ref{upperbound}.

\begin{lemma}\label{anulus}
If $D$ is a connected, $R2$-reduced diagram of a link $K$ and $D$ contains a twist region that is not a disk, then $D$ is equivalent, via a planar isotopy, to the standard diagram of the $(2,n)$ torus link.

\end{lemma}

\begin{proof}
Let $D$ be a connected diagram of a link $K$. Recall that a twist region of $D$ is either the neighborhood of a crossing of $D$ that is not incident to a bigon region or it is a connected component of $\tau$ where $\tau$ is an open regular neighborhood of all bigon regions of $D$ in the sphere of projection.

Suppose $D$ contains a twist region $T$ that is not a disk. Since a neighborhood of a crossing of $D$ is a disk, then $T$ is a connected component of $\tau$. Suppose that $D$ has the property that there are at most two bigon regions incident to every crossing of $D$ and that no two bigon regions share a common edge. In this case, $T$ deformation retracts onto a connected, compact $1$-manifold $\alpha_T$ and $T$ is homeomorphic to an open neighborhood of $\alpha_T$. The 1-manifold $\alpha_T$ can be constructed by choosing a point in the interior of each bigon region an connecting pairs of points in adjacent bigon regions via an arc that is contained in the union of the two bigon regions. See Figure \ref{fig:1.eps}. If $\alpha_T$ is an arc, then $T$ is homeomorphic to a disk, a contradiction to how we chose $T$. Hence, $\alpha_T$ is a circle and $T$ is homeomorphic to an open annulus. However, if $T$ is an open annulus, then $D$ is the union of the boundaries of the bigon regions contained in $T$. Since $D$ is $R2$-reduced, then $D$ is alternating and is the standard diagram of the $(2,n)$ torus link.

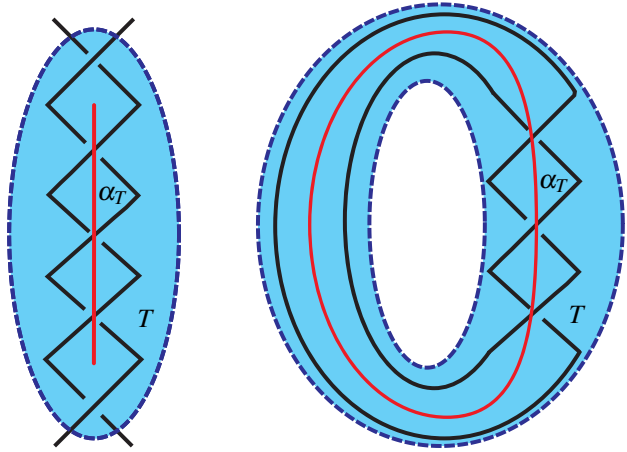
\begin{figure}[ht]
\labellist \small\hair 2pt
\pinlabel {$T$} [b] at 210 175
\pinlabel {$\alpha_{T}$} [b] at 160 360
\pinlabel {$T$} [b] at 860 180
\pinlabel {$\alpha_{T}$} [b] at 825 378
\endlabellist
\includegraphics[scale=.25]{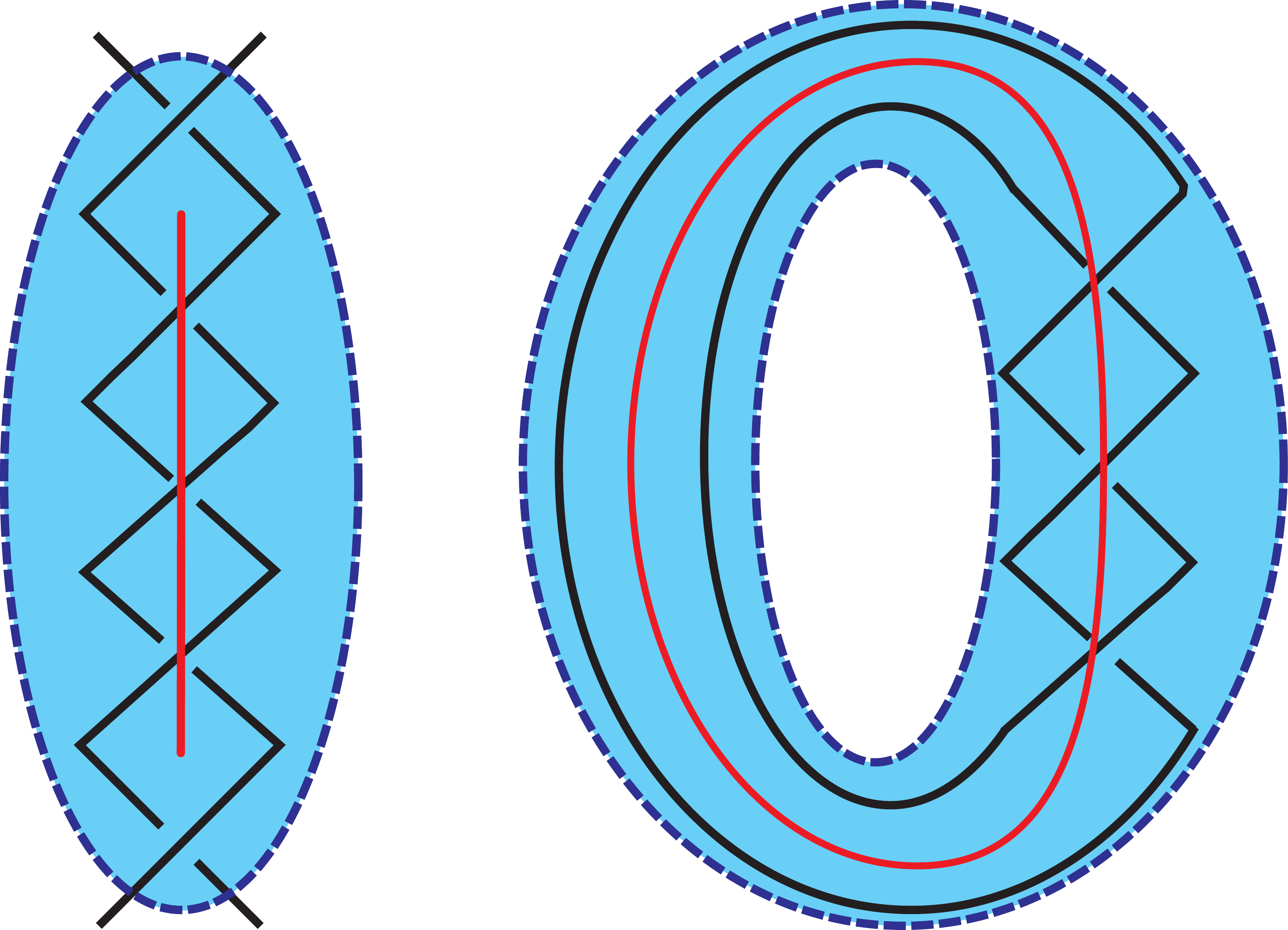}
\caption{In the left figure, $\alpha_T$ is an arc and $T$ is a disk. In the right figure, $\alpha_T$ is a circle, $T$ is an open annulus and $D$ is the standard diagram of the $(2,3)$ torus link.}\label{fig:1.eps}
\end{figure}

%\begin{figure}[h]
%\centering \scalebox{.25}{\includegraphics{1r.eps}}
%\caption{In the left figure, $\alpha_T$ is an arc and $T$ is a disk. %In the right figure, $\alpha_T$ is a circle, $T$ is an open annulus %and $D$ is the standard diagram of the $(2,3)$ torus link.}%\label{fig:1.eps}
%\end{figure}

We now consider the case when there is a crossing of $D$ which is incident to at least three distinct bigon regions. However, this can only occur when $D$ is a diagram with two vertices, four edges and four bigon regions. In this case $T$ is the entire sphere of projection. Since $D$ is $R2$-reduced, $D$ is alternating and is the standard diagram of the Hopf link. See Figure \ref{fig:2.eps}.

\begin{figure}[h]
\centering \scalebox{.25}{\includegraphics{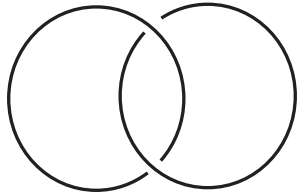}}
\caption{The standard diagram of the Hopf Link.}\label{fig:2.eps}
\end{figure}

Suppose that $D$ has the property that there are at most two bigon regions incident to every crossing of $D$ and that there exist two bigon regions $A$ and $B$ which share a common edge. See Figure \ref{fig:3.eps}. If $T$ is a twist region which is disjoint from all pairs of bigon regions that share a common edge, then the result follows from the previous arguments. With out loss of generality, suppose $T$ is not disjoint from $A\cup B$. Since $A$ and $B$ share an edge, then $T$ contains $A \cup B$. If $T$ contains a third bigon region, then one of the crossings contained in the boundary of $A$ and $B$ would be incident to three bigon regions, a contradiction to how we chose $D$. Thus, $T$ is an open neighborhood of $A\cup B$ which is a disk, a contradiction to how we chose $T$.
\end{proof}

\begin{figure}[ht]
\labellist \small\hair 2pt
\pinlabel {$B$} [b] at 300 230
\pinlabel {$A$} [b] at 380 300
\endlabellist
\includegraphics[scale=.25]{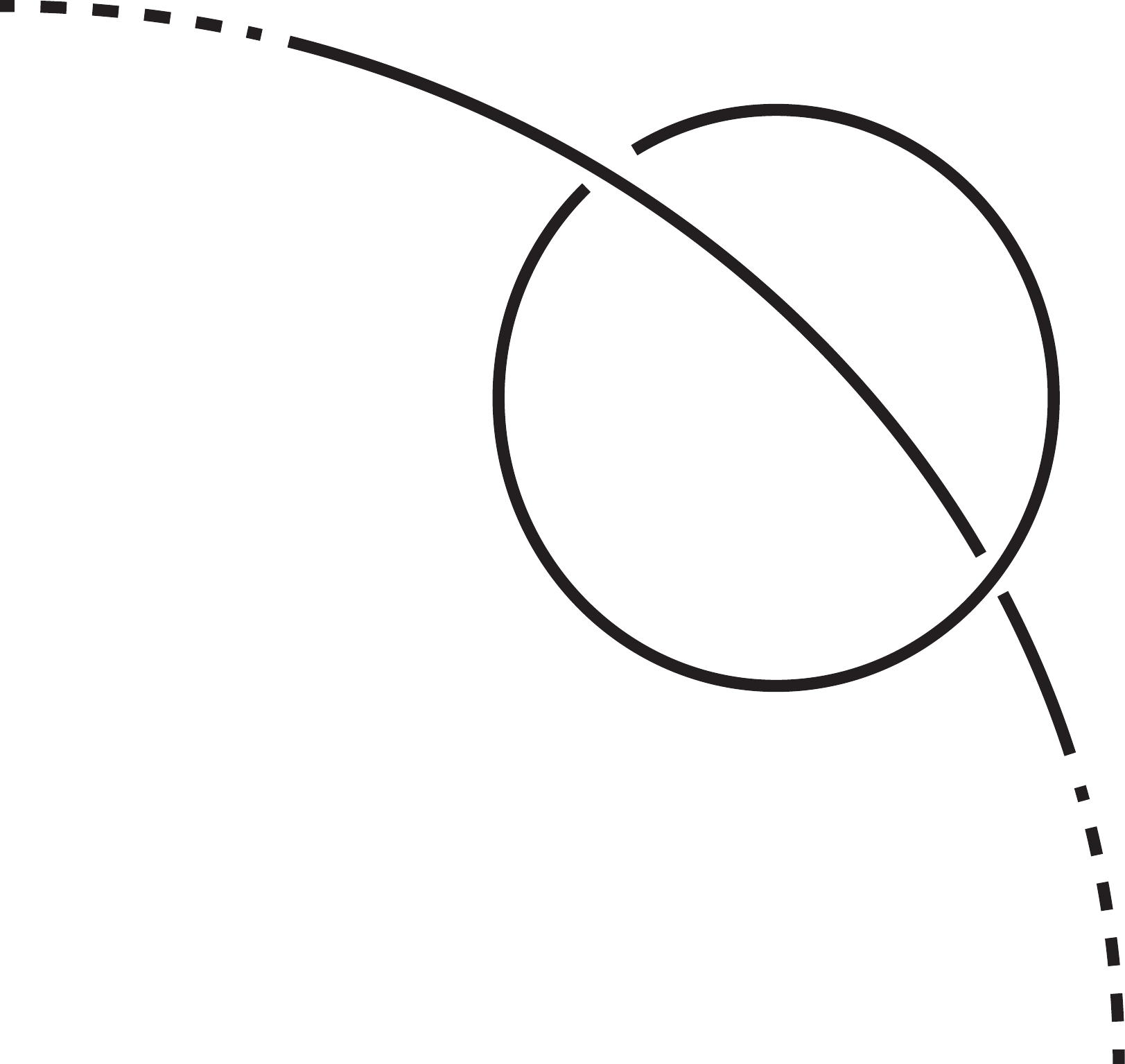}
\caption{A twist region containing two bigon regions, $A$ and $B$, that share an edge.}
\label{fig:3.eps}
\end{figure}

%\begin{figure}[h]
%\centering \scalebox{.25}{\includegraphics{3r.eps}}
%\caption{A twist region containing two bigon regions, $A$ and $B%$, that share an edge.}\label{fig:3.eps}
%\end{figure}

\begin{theorem}\label{prime}\cite{Menasco}
If $K$ is a link with reduced, alternating diagram $D$, then $K$ is non-split if and only if $D$ is connected and $K$ is prime if and only if $D$ is connected and prime.
\end{theorem}

The proof of Lemma \ref{upperbound} follows the general outline of the proofs of Theorem 4 and Corollary 5 in \cite{Blair1}. However, significant modification and additional work must be done to control the number of twist regions in $G$.

\begin{lemma}\label{upperbound}

Given any connected, $R2$-reduced, prime, non-alternating diagram, $D$, of a link $K$, we can augment the diagram by adding a single unknotted component, $U$, so that $K\cup U$ is a hyperbolic alternating augmentation of $D$ and $t(G)\leq 5t(D)$.

\end{lemma}

\begin{proof}

Let $D$ be an $R2$-reduced, non-alternating diagram of a non-split, prime link. By the proof of Lemma \ref{BlairUnlink} and Remark 1, we can create the link $K\cup \ob{U}$ such that $\ob{U}$ is an unlink, $\ob{U}$ projects to a disjoint collection of simple closed curves in the sphere of projection for $D$ and the corresponding diagram, $\ob{G}$, of $K\cup \ob{U}$ is alternating. Moreover, $\ob{G}$ is connected and the projection of $\ob{U}$ intersects each non-alternating edge of $D$ exactly once, but is otherwise disjoint from $D$. Hence, we can assume that the projection of $\ob{U}$ is disjoint from the twist regions of $D$.

Let $cl(\ob{G}\setminus D) =\bigcup_{1\leq i\leq n}C_i=A_u$ where each $C_i$ is a simple closed curve in the sphere of projection. If $n = 1$, then $\ob{U}$ consists of a single component, as desired, and we can procede to showing $t(G)\leq 5t(D)$ and $K\cup U$ is hyperbolic. Assume $n>1$. Let $\mu$ be an arc properly embedded in the closure of a path component of $S^2\setminus A_u$ which has end points in distinct boundary components, meets the edges of $D$ transversely and is disjoint from the vertices of $D$. Let $\phi(\mu)$ be the number of intersections between $\mu$ and $D$. Choose $\mu$ in such a way that it minimizes $\phi(\mu)$ over \emph{all possible choices of $\mu$ over all path components of $S^2\setminus A_u$}.

\begin{claim}\label{dijionttwist}

The arc $\mu$ can be chosen so that it does not pass through twist regions of $D$.

\end{claim}

\begin{proof}

    Since $\mu$ was chosen to be disjoint from the vertices of $D$, if $\mu$ intersects a twist region of $D$, then it intersects a bigon region of $D$. If $\mu$ intersects a bigon region of $D$ in an arc that has both endpoints on the same edge of the bigon region, then, after choosing an outermost such arc of intersection between $\mu$ and the bigon region, there is a planar isotopy of $\mu$ which eliminates an arc of intersection between $\mu$ and the bigon region. After this isotopy, $\mu$ intersects $D$ in fewer points, a contradiction to our choice of $\mu$. See Figure \ref{fig:NotInTwistRegion1.eps}.

\begin{figure}[h]
\centering \scalebox{.3}{\includegraphics{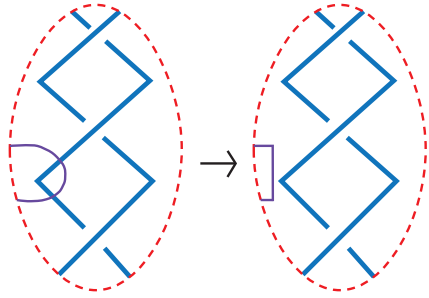}}
\caption{An isotopy that eliminates intersections between $\mu$ and $D$.}\label{fig:NotInTwistRegion1.eps}
\end{figure}

    Hence, if $\mu$ intersects a bigon region of $D$ in a non-empty collection of arcs, then each of these arcs has one endpoint in each of the two distinct boundary edges of the bigon region. Choose an outermost such arc and perform a planar isotopy of $\mu$ which removes this arc of intersection with the bigon region, but preserves $\phi(\mu)$. Repeat this process for each consecutive bigon region, always pushing arcs in the same direction. See Figure \ref{fig:NotInTwistRegion2.eps}. Note that, by Lemma \ref{anulus}, all twist regions of $D$ are disks, since otherwise $D$ would be alternating. Hence, this sequence of isotopies terminates when $\mu$ is isotoped to be disjoint from all twist regions.

    %Repeat this process for each consecutive bigon region, always pushing arcs in the same direction, until $\mu$ is disjoint from all twist regions. See Figure \ref{fig:NotInTwistRegion2.eps}. Note that, by Lemma \ref{anulus}, all twist regions of $D$ are disks, since otherwise $D$ would be alternating. Hence, this sequence of isotopies terminates.

\begin{figure}[h]
\centering \scalebox{.3}{\includegraphics{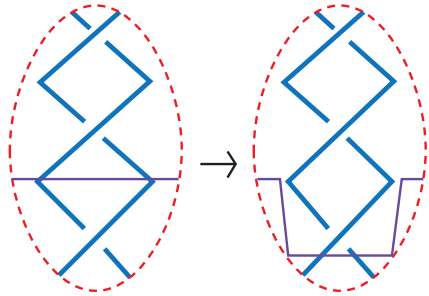}}
\caption{An isotopy that eliminates intersections between $\mu$ and a bigon region.}\label{fig:NotInTwistRegion2.eps}
\end{figure}

Since all of the planar isotopies of $\mu$ described in this claim are supported in a neighborhood of the twist regions of $D$ and $A_u$ is disjoint from the twist regions of $D$, then the interior of $\mu$ remains disjoint from $A_u$ during these isotopies. Hence, after these isotopies, $\mu$ remains a properly embedded arc in a closure of a path-component of $S^2\setminus A_u$ and $\phi(\mu)$ has not increased. Therefore, we can always assume that $\mu$ is disjoint from the twist regions of $D$.

\end{proof}

 By the previous claim, we can choose $\mu$ such that $\mu$ is disjoint from all twist regions and $\phi(\mu)$ is minimal. Let $H$ be the closure of the path component of $S^2\setminus A_u$ such that $\mu$ is properly embedded in $H$. Let $C_i$ and $C_j$ be the distinct boundary components of $H$ which have non-trivial intersection with $\mu$.

In the arguments that follow, we wish to understand how $\mu$ intersects $D$. Our immediate goal is to use induction and $\mu$ to create $G$, the projection of an alternating augmentation of $D$, with augmenting component $U$ which projects to a simple closed curve $A$ such that $A$ is disjoint from the twist regions of $D$ and $A$ meets each edge of $D$ at most twice.

%If $\mu$ is disjoint from $D$, then we may use a Type I move as described in Lemma \ref{TypeI} to join $C_i$ and $C_j$ into a single simple closed curve in the sphere of projection so the that resulting diagram is alternating. Hence, we will assume that $\mu$ has non-trivial intersection with $D$.

Suppose, by way of contraction, that $A_u$ and $\mu$ intersect a common edge of $D$ denoted $e$. Since $\mu$ is properly embedded in $H$, then there must be a point $a \in e\cap \mu$ and a point $b\in e\cap C_k$ where $C_k$ is some boundary component of $H$ such that $a$ and $b$ cobound a sub arc $\beta$ of $e$ such that the interior of $\beta$ is disjoint from $\mu \cup A_u$. Perform surgery on $\mu$ along the arc $\beta$ to create two new arcs $\mu_1$ and $\mu_2$. See Figure \ref{fig:Figure1.eps}.

\begin{figure}[ht]
\labellist \small\hair 2pt
\pinlabel {$\mu$} [b] at 230 240
\pinlabel {$C_{k}$} [b] at 80 230
\pinlabel {$\mu_{1}$} [b] at 690 235
\pinlabel {$\mu_{2}$} [b] at 690 -25
\pinlabel {$C_{k}$} [b] at 540 225
\endlabellist
\includegraphics[scale=.4]{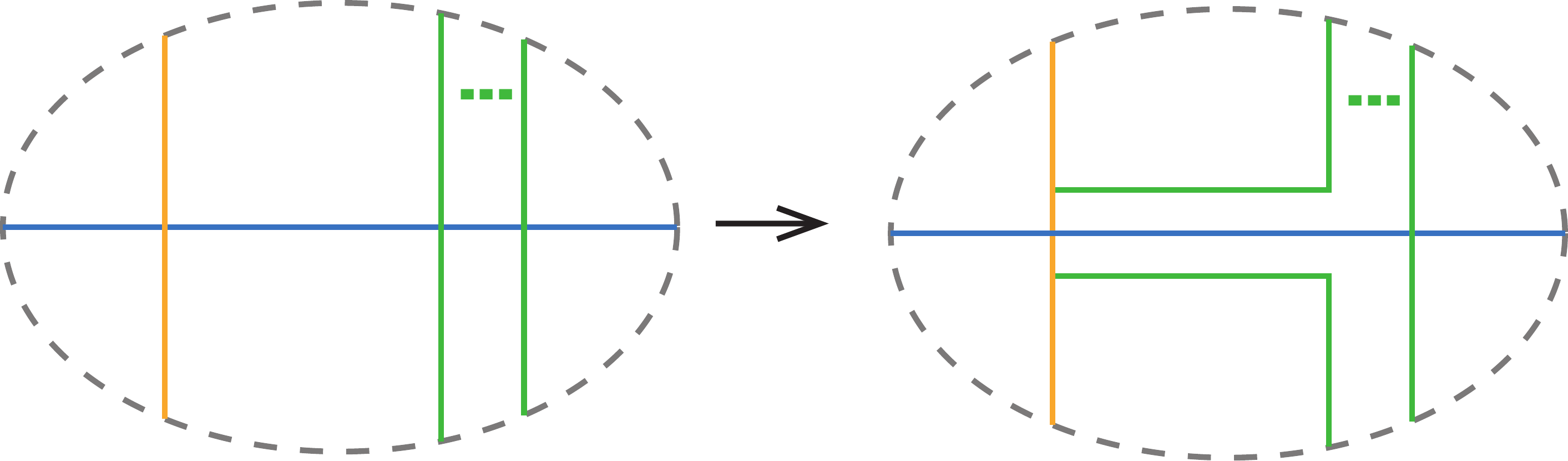}
\caption{A surgery on $\mu$ resulting in the arcs $\mu_1$ and $\mu_2$.}
\label{fig:Figure1.eps}
\end{figure}
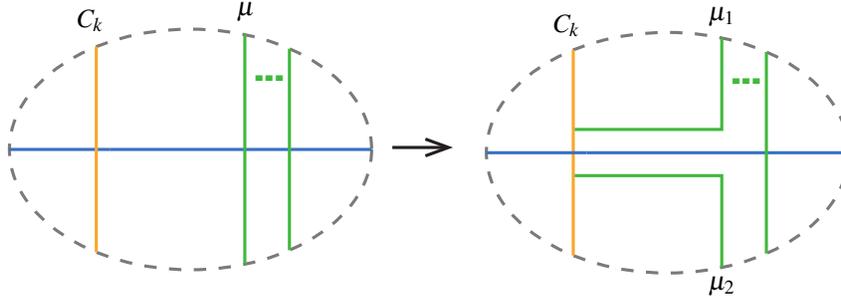

%\begin{figure}[h]
%\centering \scalebox{.4}{\includegraphics{Figure1r.eps}}
%\caption{A surgery on $\mu$ resulting in the arcs $\mu_1$ and $\mu_2$.}\label{fig:Figure1.eps}
%\end{figure}

Without loss of generality, assume $\mu_1$ has an endpoint on $C_i$ and $\mu_2$ has an endpoint on $C_j$.

Since $C_k$ and $\mu$ are disjoint from all twist regions of $D$, then $\mu_1$ and $\mu_2$ are disjoint from all twist regions of $D$. Note that $\phi(\mu)=\phi(\mu_1)+\phi(\mu_2)-1$

    If $C_i=C_k$, then $\mu_2$ connects distinct components of $A_u$ and intersects $D$ in fewer points than $\mu$.

   If $C_j=C_k$, then $\mu_1$ connects distinct components of $A_u$ and intersects $D$ in fewer points than $\mu$.

If $C_k\neq C_i$ and $C_k \neq C_j$, then, since $\phi(\mu)=\phi(\mu_1)+\phi(\mu_2)-1$, $\mu_1$ has strictly fewer points of intersection with D than $\mu$ and $\mu_1$ connects 2 distinct components of $A_u$.

In each case, we have constructed a properly embedded arc in $H$ which connects distinct boundary components of $H$ and meets $D$ in fewer points than $\mu$, a contradiction. Thus, $\mu$ is disjoint from any edge of $D$ that meets $A_u$.

We proceed by showing that $\mu$ cannot intersect a single edge of $D$ more than once.

Suppose $\mu$ intersects an edge, $e$, of $D$ transversely in two or more points. Perform surgery on $\mu$ along a sub arc of $e$ (see Figure \ref{fig:Figure2.eps}) joining two consecutive points of $\mu\cap e$. Let $\mu^*$ be the 1-manifold produced by this surgery. If the resulting 1-manifold $\mu^*$ has 2 components, then one component of $\mu^*$ connects $C_i$ to $C_j$. Call this component $\mu'$. The other component of $\mu^*$ is a simple closed curve. If the resulting 1-manifold $\mu^*$ is a single arc connecting $C_i$ to $C_j$, then let $\mu'=\mu^*$.

\begin{figure}[ht]
\labellist \small\hair 2pt
\pinlabel {$\mu$} [b] at 230 240
\pinlabel {$e$} [b] at -15 108
\pinlabel {$e$} [b] at 445 105
\pinlabel {$\mu^{*}$} [b] at 695 235
\endlabellist
\includegraphics[scale=.4]{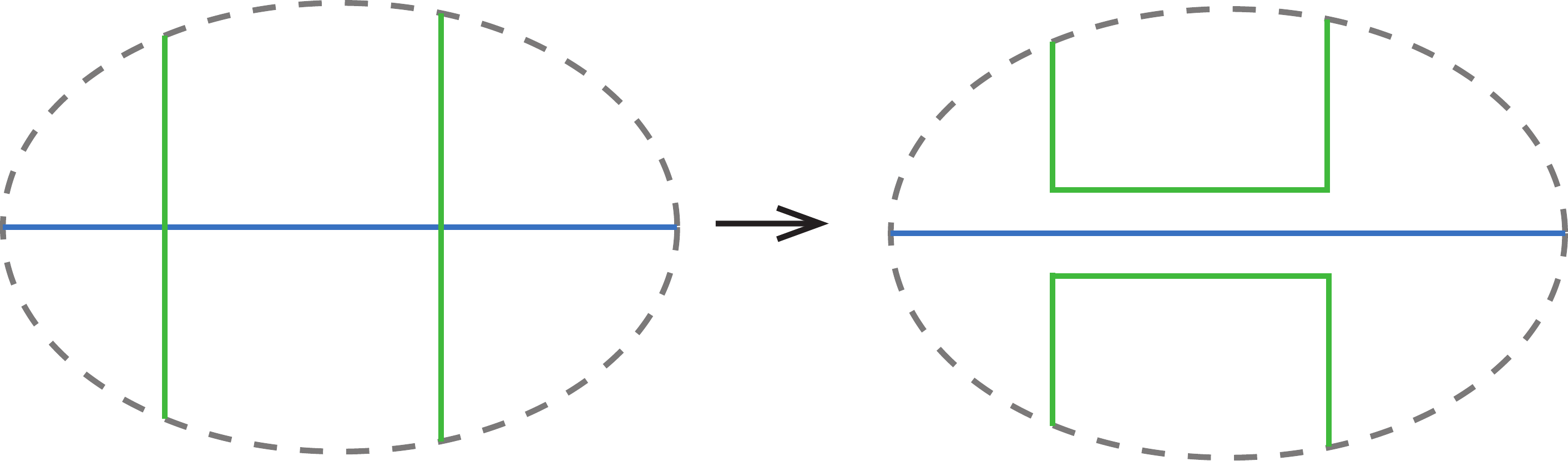}
\caption{A surgery on $\mu$ which decreases $D\cap \mu$.}
\label{fig:Figure2.eps}
\end{figure}
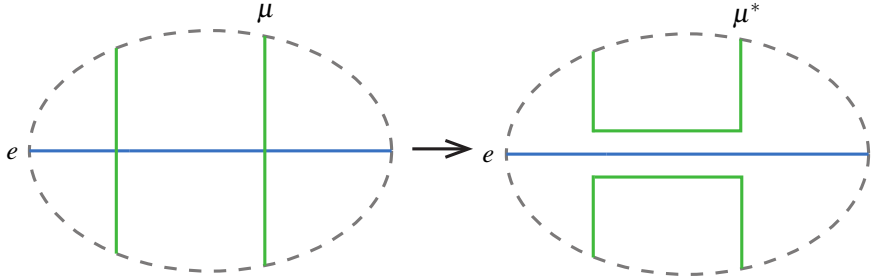

%\begin{figure}[h]
%\centering \scalebox{.4}{\includegraphics{Figure2r.eps}}
%\caption{A surgery on $\mu$ which decreases $D\cap \mu$.}%\label{fig:Figure2.eps}
%\end{figure}

In each case, there exists an arc $\mu'$ that connects $C_i$ to $C_j$ and intersects $D$ in at least 2 fewer points than $\mu$, thus, providing a contradiction to our choice of $\mu$. Hence, we can assume $\mu$ will not intersect an edge of $D$ more than once.

To summarize, we have just shown that we can pick $\mu$ so that it is disjoint from the twist regions of $D$, is disjoint from any edge of $D$ that meets $A_u$ non-trivially and intersects every edge of $D$ in at most one point.

We can propagate $C_i$ along $\mu$ using a sequence of Type II moves, as depicted in Figure \ref{fig:Type1Type2MovesFigure3.eps}, until $C^{*}_{i}$, which is the image of $C_i$ under Type II moves, and $C_j$ are incident to a common region of the resulting diagram. Call this resulting diagram $G^*$.

\begin{figure}[ht]
\labellist \small\hair 2pt
\pinlabel {$\mu$} [b] at 175 345
\pinlabel {$G$} [b] at 150 -40
\pinlabel {$G^{*}$} [b] at 570 -40
\pinlabel {$C_j$} [b] at 305 150
\pinlabel {$C_j$} [b] at 730 150
\pinlabel {$C_{i}$} [b] at 150 625
\pinlabel {$C_{i}^{*}$} [b] at 570 625
\pinlabel {$G^{**}$} [b] at 1000 -40
%\pinlabel {$e$} [b] at 445 105
%\pinlabel {$\mu^{*}$} [b] at 695 235
\endlabellist
\includegraphics[scale=.25]{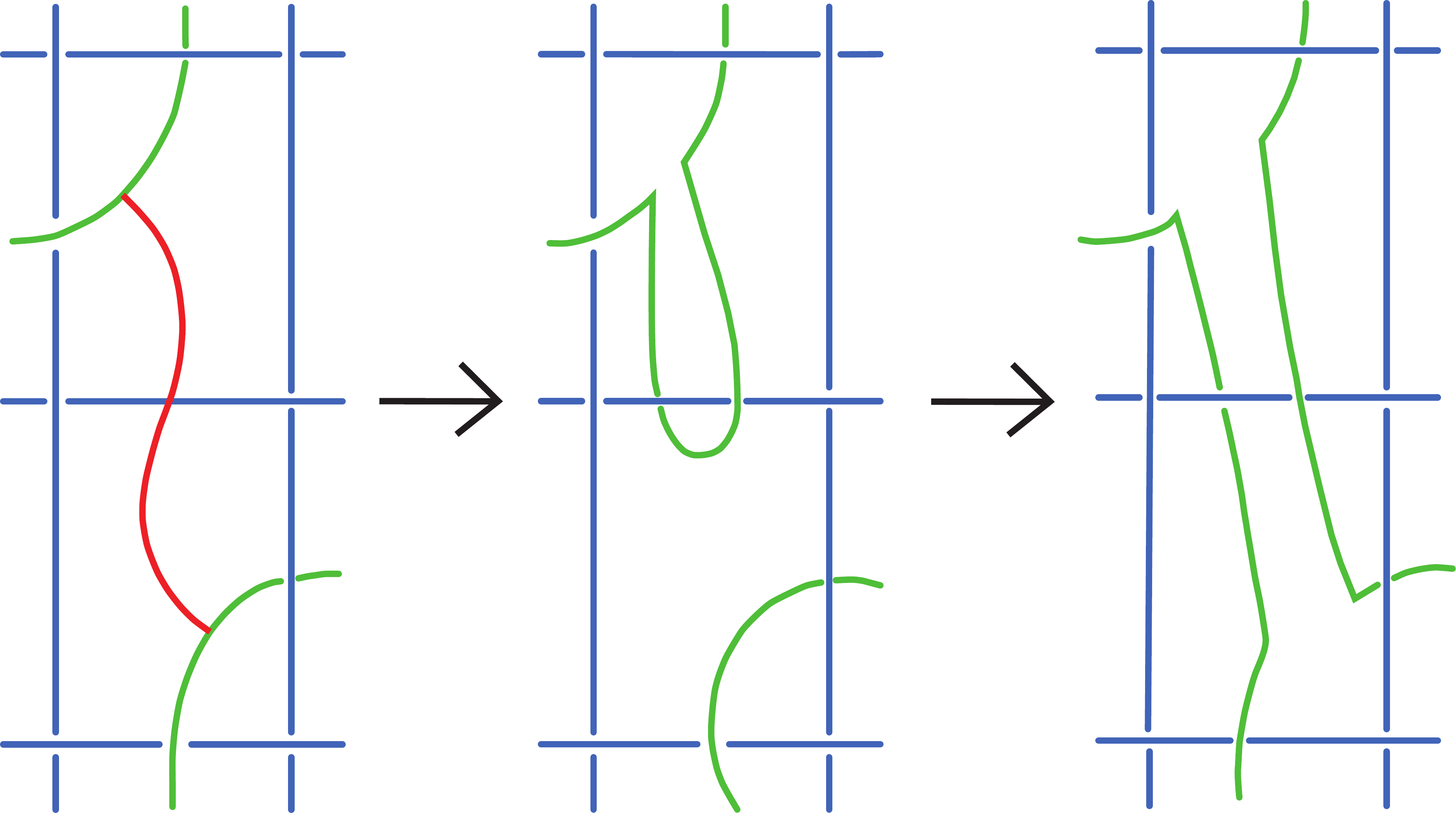}
\caption{Connecting components of $A_u$}\label{fig:Type1Type2MovesFigure3.eps}
\end{figure}
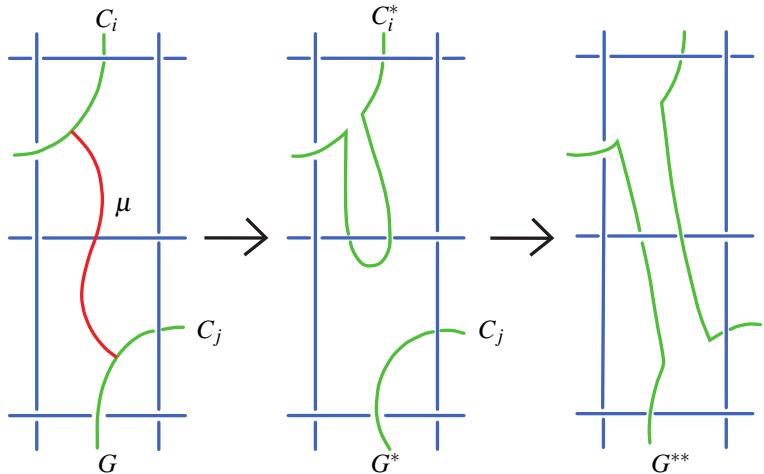

%\begin{figure}[h]
%\centering \scalebox{.25}%{\includegraphics{Type1Type2Movesr.eps}}
%\caption{Connecting components of $A_u$}\label{fig:Type1Type2MovesFigure3.eps}
%\end{figure}

%Call this projection $G^*$, which is alternating, and has components that do not go through twist regions of $D$, and only cross an edge of $D$ at most twice.

By  Lemma \ref{TypeII}, $G^*$ is alternating. Since $\mu$ is disjoint from the twist regions of $D$, then $cl(G^*\setminus D)$ is disjoint from the twist regions of $D$. Since $\mu$ meets every edge of $D$ at most once and $\mu$ is disjoint from every edge of $D$ that meets $A_u$, then every component of $cl(G^*\setminus D)$ meets an edge of $D$ at most twice.

%Since a Type II move is an isotopy of $C_i$ in $S^2$ and $\mu$ is properly embedded in $H$, then $C^*_i$ is a simple closed curve which does not intersect the other components of $cl(G\setminus D)$. Hence, $G^*$ is the projection of an alternating augmentation of $D$.

Next, we use the Type I move to connect sum the disjoint simple closed curves $C^*_i$ and $C_j$ into a single simple closed curve. Call the resulting projection $G^{**}$. Since $G^{*}$ is alternating, so is $G^{**}$, by Lemma \ref{TypeI}. Additionally, every component of $cl(G^{**}\setminus D)$ is disjoint from the twist regions of $D$ since $\mu$ was chosen to be disjoint from the twist regions of $D$. Finally, each component of $cl(G^{**}\setminus D)$ crosses every edge of $D$ at most twice, since $\mu$ was chosen to intersect each edge of $D$ at most once and since $\mu$ was chosen to be disjoint from any edge of $D$ that meets $A_u$. Hence, $G^{**}$ is an alternating link diagram containing $D$ and $cl(G^{**}\setminus D)$ consists of one fewer loop then $cl(\ob{G}\setminus D)$. Repeat this process to create $G$ the projection of an alternating augmentation of $D$ with exactly one augmenting component $U$ which projects to a simple closed curve $A$ in the sphere of projection. It follows from the above arguments that $A$ will be disjoint from the twist regions of $D$ and $A$ will meet each edge of $D$ at most twice.

Since $A$ is disjoint from all twist regions of $D$ and $A$ never intersects itself, then $t(G)\leq t(D) +i(A,D)$ where $i(A,D)$ is the number of intersections between $A$ and $D$. Since $A$ is disjoint from twist regions of $D$ and crosses each edge of $D$ not contained in a twist region of $D$ at most twice, then $i(A,D)$ is less than or equal to twice the number of edges of $D$ not contained in twist regions of $D$. However, the number of edges of $D$ not contained in twist regions of $D$ is less than or equal to $2t(D)$ since the union of the edges of $D$ not contained in twist regions of $D$ together with the twist regions of $D$ form a finite 4-valent graph in the sphere with each twist region corresponding to a single vertex. Thus,

$$t(G)\leq t(D) + 2(2t(D))=5t(D).$$

It remains to be shown that $G$ is the diagram of a hyperbolic link. By Theorem \ref{Menasco}, it is sufficient to show that the link $K\cup U$ is non-split, prime and not a $(2,q)$ torus link. Note that since we have assumed that $D$ is reduced and $A$ projects to a simple closed curve in the sphere of projection, then $G$ is reduced. Since $G$ is a reduced, alternating diagram, by Theorem \ref{prime}, $G$ is the diagram of a hyperbolic link if $G$ is prime, connected and not a diagram of a $(2,q)$ torus link.

Since we have assumed that $D$ is connected and $A$ is a simple closed curve in the sphere of projection with a non-trivial intersection with $D$, then $G$ is connected.

In search of a contradiction, suppose that $G$ is not prime. Hence, there exists a simple closed curve $\beta$ in the sphere of projection such that $\beta$ intersects $G$ transversely in exactly two points neither of which is a vertex of $G$ and both disks that $\beta$ bounds in the sphere of projection contain vertices of $G$. Let $E_1$ and $E_2$ denote the disks that $\beta$ bounds. Since both $D$ and $A$ are closed curves in the sphere of projection, then $\beta\cap G\subset D$ or $\beta\cap G\subset A$. If $\beta\cap G\subset A$, then, since $A$ is a simple closed curve and $D$ is connected, it is impossible for both $E_1$ and $E_2$ to contain a vertex of $G$. Hence, we can assume that $\beta\cap G\subset D$. Since $A$ is a connected subset of the sphere of projection which is disjoint from $\beta$, then $A\subset E_1$ or $A\subset E_2$. Without loss of generality, assume that $A\subset E_1$. Since both $E_1$ and $E_2$ must contain a crossing of $G$, then $E_2$ must contain a crossing of $D$. If $E_1$ also contains a crossing of $D$, then $D$ is not prime, a contradiction to our choice of $D$. Thus, the only crossings of $G$ contained in $E_1$ are points of intersection between $A$ and $D$. Since $E_1$ contains no crossings of $D$ and $\beta=\partial E_1$ meets $D$ in exactly two points, then $E_1\cap D$ consists of exactly one subarc of exactly one edge of $D$. Hence, $A$ meets at most one edge of $D$.

Since $G$ is alternating and $A$ meets at most one edge of $D$, then all but at most one edge of $D$ is non-alternating. Since every link diagram has an even number of non-alternating edges, then $D$ must be an alternating diagram. This is a contradiction to the fact that we choose $D$ to be non-alternating. Thus, $G$ is prime.

In search of a contradiction, suppose that $G$ is the diagram of a $(2,q)$ torus link. Recall that we have already established that $G$ is a prime, connected, reduced, alternating diagram. However, the only prime, connected, reduced, alternating diagram of a link which is not hyperbolic is the standard diagram of the $(2,q)$ torus link with $q$ crossings. This implies $G$ is the standard diagram of the $(2,q)$ torus link with $q$ crossings. Hence, both $D$ and $A$ are simple closed curves in the sphere of projection. However, this is a contradiction to $D$ being non-alternating.

Thus, $K\cup U$ is hyperbolic, completing the proof.

\end{proof}

\begin{theorem}
Given any prime, non-alternating knot $K$

$$V_3(t(K)-2)\leq AltVol(K)\leq 10V_3(5t(K)-1)$$
\end{theorem}

\begin{proof}

Let $K$ be a prime, non-alternating knot and let $G$ be a reduced diagram of an alternating augmentation of $K$ whose volume is equal to $AltVol(K)$ and let $D$ be the diagram of $K$ that results from considering $G$ and ignoring the augmenting component. By Lemma \ref{lowerbound}, $t(D)\leq t(G)$. Since $G$ is the reduced alternating diagram of a hyperbolic link, then, by Theorem \ref{prime}, $G$ must be a prime diagram. Since $G$ is a reduced, prime, alternating diagram of a hyperbolic link, then, By Theorem \ref{Lackenby}, $V_3(t(G)-2)\leq AltVol(K)$. Hence,

$$V_3(t(K)-2)\leq V_3(t(D)-2)\leq V_3(t(G)-2) \leq AltVol(K).$$

Let $D^*$ be a diagram of $K$ such that $t(D^*)=t(K)$. Since $K$ is a knot, every diagram of $K$ is connected. Since flypes and the type II Reidermeister move that decreases crossing number can only decrease the number of twist regions in a link diagram, then we can assume that $D^*$ is reduced and $R2$-reduced. Since a twist number minimizing diagram of a prime knot is always a prime diagram, then $D^*$ is prime. Since $K$ is non-alternating, then $D^*$ is non-alternating. Let $U$ be the augmenting component and let $G$ be the diagram corresponding to the alternating augmentation of $D^*$ constructed in the proof of Lemma \ref{upperbound}. By Lemma \ref{upperbound}, we know that $G$ is alternating, $t(G)\leq 5t(D^*)=5t(K)$ and $K\cup U$ is a hyperbolic link. By Theorem \ref{Lackenby}, $vol(S^3\setminus (K\cup U))\leq 10V_3(t(G)-1)$. Since $AltVol(K)\leq vol(S^3\setminus (K\cup U))$, then $AltVol(K)\leq 10V_3(t(G)-1)$. Hence,

%We don't know that $G$ is a diagram of a hyperbolic link!!!

$$AltVol(K)\leq 10V_3(5t(K)-1)$$

\end{proof}

\section{Acknowledgements} The authors would like to thank Yo'av Rieck for many helpful discussions. We are grateful to the Undergraduate Research Group program organized by the California State University Alliance Preparing Undergraduates through Mentoring towards PhDs for helping to support this research project. The authors were partially supported by the National Science Foundation grants DMS--1247679 and DMS--1821254.

\end{document}